\documentclass[10pt,reqno]{amsart}
\usepackage{a4}
\newtheorem{theorem}{Theorem}[section]
\newtheorem{lemma}[theorem]{Lemma}
\newtheorem{proposition}[theorem]{Proposition}
\newtheorem{corollary}[theorem]{Corollary}
\theoremstyle{definition}
\newtheorem{definition}[theorem]{Definition}

\theoremstyle{remark}
\newtheorem{remark}[theorem]{Remark}
\numberwithin{equation}{section}

\usepackage[all,cmtip]{xy}
\usepackage{centernot}
\usepackage{enumitem}
\usepackage{hyperref}
\newtheorem{question}[theorem]{Question}

\begin{document}

\title[L\'evy-Khinchin decomposition]{On the L\'evy-Khinchin decomposition of generating functionals}

\author{Uwe Franz}
\address{U.F., D\'epartement de math\'ematiques de Besan\c{c}on,
Universit\'e de Franche-Comt\'e, 16, route de Gray, 25 030
Besan\c{c}on cedex, France}
\email{uwe.franz@univ-fcomte.fr}
\urladdr{http://lmb.univ-fcomte.fr/uwe-franz}
\thanks{U.F.\ was supported by an ANR Project OSQPI (ANR-11-BS01-0008) and by the Alfried Krupp Wissenschaftskolleg Greifswald as senior fellow.} 

\author{Malte Gerhold}
\address{M.G., Institut f\"ur Mathematik und Informatik \\ 
Ernst Moritz Arndt Universit\"at 
Greifswald\\Walther-Rathenau-Stra\ss{}e 47 \\
17487 Greifswald \\ Germany}
\email{mgerhold@uni-greifswald.de}
\urladdr{www.math-inf.uni-greifswald.de/index.php/mitarbeiter/282-malte-gerhold}

\author{Andreas Thom}
\address{A.T., FB Mathematik, Institut f\"ur Geometrie\\ 
TU Dresden \\ 01062 Dresden}
\email{andreas.thom@tu-dresden.de}
\urladdr{http://tu-dresden.de/mathematik/geometrie/thom}
\thanks{A.T.\ was supported by ERC Starting Grant No.\ 277728.}



\begin{abstract}
We study several sufficient conditions for the existence of a L\'evy-Khinchin decomposition of generating functionals on unital involutive algebras with a fixed character. We show that none of these conditions are equivalent and we show that such a decomposition does not always exist.
\end{abstract}

\maketitle

\section{Introduction}

Convolution semigroups of probability measures on locally compact abelian groups have semigroups of positive definite functions on the dual group as Fourier transform and can therefore be classified by conditionally positive definite functions on the dual group. In these classifications, first obtained for the real line by Khinchin and L\'evy in the 1930's, the conditionally positive definite functions are written as sum of a quadratic or Gaussian part and an integral part which does not contain a (non-degenerate) Gaussian part. We will call such a decomposition a L\'evy-Khinchin decomposition, see Definition \ref{def-lk}. There exist similar classifications and decompositions on general locally compact abelian groups, cf.\ \cite{guichardet72a,parthasarathy+schmidt72}.

In the characterization of convolution semigroups of probability measures on
possibly noncommutative Lie groups, Hunt \cite{hunt56} replaced
conditional positive definite functions by generating functionals or
generators of the associated Markov semigroup. They are again a sum of a quadratic or Gaussian part and an integral part that corresponds to the jumps of the associated L\'evy process.

Sch\"urmann \cite{schuermann90} \cite[Chapter 5]{schuermann93} investigated if such a decomposition is also possible in the still more general setting of L\'evy processes on involutive bialgebras. Here one would
like to characterize generating functionals, i.e., linear functionals on a unital $*$-algebra that are hermitian, positive on the kernel of a character $\varepsilon\colon A\to
\mathbb{C}$, and vanish on the unit, see Definition \ref{def-cpf}. Sch\"urmann introduced several cohomological
conditions on such pairs $(A,\varepsilon)$ that guarantee that any generating
functional on $(A,\varepsilon)$ can be decomposed into a Gaussian part and a
purely non-Gaussian part. By Sch\"urmann's generalization of Schoenberg's correspondence, generating functionals on involutive bialgebras are in one-to-one correspondence with convolution semigroups of states and therefore classify L\'evy processes, see \cite{schuermann93}. Note that generating functionals are a generalization of conditionally positive functions on groups (which are also known --- up to the sign --- as functions of negative type). If $A=\mathbb{C}G$ is the $*$-algebra of a group $G$ and $\varepsilon$ the linear extension of the trivial representation, then a hermitian functional $\psi\colon\mathbb{C}G\to\mathbb{C}$ with normalization $\psi(\mathbf{1})=0$ is a generating functional if and only if $\psi|_G$ is conditionally positive (or $-\psi|_G$ is a function of negative type).

Sch\"urmann showed that the L\'evy-Khinchin decomposition of a generating functional is always possible if $A$ is a commutative $*$-bialgebra, and that it is
also possible for any generating functional on the Brown-Glockner-von
Waldenfels $*$-algebra generated by $n^2$ elements satisfying the unitarity
relations.

In this paper we continue Sch\"urmann's study and show that none of the sufficient cohomological conditions appearing in his work are equivalent and that none of them are necessary for the existence of a L\'evy-Khinchin decomposition of arbitrary generating functionals on a given pair $(A,\varepsilon)$. We also show that there exist pairs $(A,\varepsilon)$ with generating functionals that do not admit such a decomposition.

Our approach is based on an exact sequence obtained by Netzer and Thom
\cite{netzer+thom13} for group algebras, which allows to characterize the
existence and uniqueness of a generating functional for a given cocycle in
terms of the first and second homology groups, see Theorem \ref{nt-sequence},
Lemma \ref{nt-lemma}, and the discussion in Remark \ref{nt-remark}.

The paper is organized as follows. In Section \ref{sec-prelim} we recall the relevant definitions of generating functionals, Sch\"urmann triples, Gaussianity and L\'evy-Khinchin decompositions. In Section \ref{sec-hoch} we recall the Hochschild (co-)homology for associative algebras and state a generalization of the exact sequence from \cite[Lemma 5.6]{netzer+thom13}, as well as a dual version. This allows us to give a new answer to the question of existence and uniquess of a generating functional for a given pair $(\pi,\eta)$, where $\pi$ is a $*$-representation and $\eta$ a $\pi$-$\varepsilon$-cocycle, see Remark \ref{nt-remark}. Finally, in Section \ref{sec-examples}, we give an example of a generating functional that does not admit a L\'evy-Khinchin decomposition and we provide examples that show that the conditions (LK), (NC), (GC), (AC), and (H$^2$Z) to be introduced in Section \ref{sec-prelim} are not equivalent.

The following diagram summarizes the relations between the conditions we study in this paper:
\begin{equation}\label{diagram}
\xymatrixcolsep{3pc}\xymatrix{
&&{\rm (NC)} \ar@{=>}[dr]&&\\
{\rm (H^2Z)} \ar@{=>}[r]&{\rm (AC)} \ar@{=>}[dr] \ar@{=>}[ur]&& \hspace*{-2em}{\rm (GC)} \vee {\rm (NC)} \ar@{=>}[r]&{\rm (LK)}\\
&&{\rm (GC)} \ar@{=>}[ur]&&
}
\end{equation}
None of the converse implications holds in general. Our counter-examples are constructed on group algebras, except for Example \ref{exa-not2=0}, which shows that ${\rm (H^2Z)}$ and (AC) are not equivalent. We do not know if ${\rm (H^2Z)}$ and (AC) are equivalent under additional assumptions that are verified by group algebras, such as the existence of a faithful state.

In this paper we call the decomposition of a conditionally positive function or a generating functional into a Gaussian part and a purely non-Gaussian part a
\emph{L\'evy-Khinchin decomposition}. Such a decomposition is related to the decomposition of the associated L\'evy processes into a Gaussian part and a jump part, which is known as \emph{L\'evy-It\^o decomposition} in probability theory. While the classification and the
decomposition of conditionally positive functions or generating functionals
can be studied using only the $*$-algebra structure of $A$ and the character
$\varepsilon\colon A\to\mathbb{C}$, the reconstruction and decomposition of the
associated L\'evy processes depends also on the coalgebra structure, and will
be studied elsewhere.

\section{Generating functionals, Sch\"urmann triples, Gaussianity, etc.}\label{sec-prelim}

Throughout this paper, $A$ will be a unital associative involutive algebra
over the field of complex numbers and $\varepsilon\colon A\to\mathbb{C}$ a non-zero
$*$-homomorphism (also called a \emph{character}).

\begin{definition}\label{def-cpf}
We say that a linear functional $\psi\colon A\to\mathbb{C}$ is a \emph{generating functional on $(A,\varepsilon)$} if
\begin{enumerate}[label=\textnormal{(\roman*)}]
\item $\psi(1)=0$;
\item $\psi$ is hermitian, i.e., $\psi(a^*)=\overline{\psi(a)}$
  for $a\in A$;
\item $\psi$ is positive on ${\rm ker}(\varepsilon)$, i.e.,
  $\psi(a^*a)\ge 0$ for $a\in{\rm ker}(\varepsilon)$.
\end{enumerate}
\end{definition}

For a pre-Hilbert space $D$ we denote by $L(D)$ the $*$-algebra of adjointable
linear operators on $D$. See Section \ref{sec-hoch} for the definition of cocycles and coboundaries.
\begin{definition}
A triple $(\pi\colon A\to L(D),\eta\colon A\to D,\psi\colon A\to \mathbb{C})$ of linear maps is
called a \emph{Sch\"urmann triple on $(A,\varepsilon)$ over $D$} if
\begin{enumerate}[label=\textnormal{(\roman*)}]
\item $\pi$ is a unital $*$-representation,
\item $\eta$ is a $\pi$-$\varepsilon$-cocycle, i.e., we have
\[
\eta(ab) = \pi(a)\eta(b) + \eta(a)\varepsilon(b), \qquad a,b\in A,
\]
\item $\psi$ is a hermitian linear functional that has
  \[
  A\otimes A \ni (a\otimes b) \mapsto
  -\langle\eta(a^*),\eta(b)\rangle\in \mathbb{C}
\]
as coboundary, i.e., we have
\[
\varepsilon(a)\psi(b)-\psi(ab)+\psi(a)\varepsilon(b) = -\langle\eta(a^*),\eta(b)\rangle, \qquad a,b\in A.
\]

\end{enumerate}
\end{definition}
One can show that the hermitian functional $\psi$ in a Sch\"urmann triple is a generating functional. We call a Sch\"urmann triple \emph{surjective}, if $\eta\colon A\to D$ is surjective.

We will denote the linear map $A\otimes A \ni (a\otimes b) \mapsto
\langle\eta(a^*),\eta(b)\rangle\in \mathbb{C}$ by $\mathcal{L}(\eta)$. 

From a given generating functional $\psi\colon A\to\mathbb{C}$ one can
construct a Sch\"urmann triple $(\pi,\eta,\psi)$ via a GNS-type construction,
see \cite[Theorem 2.3.4]{schuermann93} or \cite[Section 4.4]{franz+schott99}.

A central problem in our paper is to determine for a given pair $(\pi,\eta)$ of a unital $*$-representation $\pi$ and a $\pi$-$\varepsilon$-cocycle $\eta$, if there exists a functional $\psi$ that makes $(\pi,\eta,\psi)$ a Sch\"urmann triple. Note that $(\pi,\eta)$ almost determines $\psi$, if the latter exists. More precisely, if $(\pi,\eta,\psi)$ is a Sch\"urmann triple on $(A,\varepsilon)$ and $\psi'$ is a Hermitian linear functional, then $(\pi,\eta,\psi')$ is a Sch\"urmann triple if and only if $d:=\psi-\psi'$ is a \emph{derivation}, i.e., $d(ab)=\varepsilon(a)d(b)+d(a)\varepsilon(b)$ for $a,b\in A$.

We use the notation $K_1$ for the kernel of $\varepsilon$ and define furthermore
\[
K_n= \operatorname{span}\{a_1\cdots a_n: a_1,\ldots,a_n\in K_1\}
\]
for $n\ge 2$. Since $\varepsilon$ is a $*$-homomorphism, we get a descending chain of $*$-ideals. In particular, we have $K_{n+1}\subseteq K_n$ for $n\in\mathbb{N}$.

\begin{definition}
A $\pi$-$\varepsilon$-cocycle on $(A,\varepsilon)$ is called \emph{Gaussian}
if it is a derivation, i.e., if
\[
\eta(ab) = \varepsilon(a)\eta(b) + \eta(b)\varepsilon(b)
\]
for $a,b\in A$.

A generating functional $\psi$ on
$(A,\varepsilon)$ is called \emph{Gaussian} if $\psi|_{K_3}=0$. This terminology is a natural generalization of the classical case. A L\'evy process with values in Euclidean space or more generally a Lie group is called Gaussian if the measure in the integral term in Hunt's formula vanishes, i.e., if its generator is a second order differential operator. In this case the counit is evaluation of a function at the origin and $K_3$ therefore consists of functions having a zero of order three at the origin. Generating functionals of L\'evy processes with values in Euclidean space or a Lie group are therefore Gaussian if and only if they vanish on $K_3$.

If $(\pi,\eta,\psi)$ is a  Sch\"urmann triple over  $(A,\varepsilon)$, then $\eta$ is Gaussian if and only if $\psi$ is Gaussian, in which case we call the Sch\"urmann triple \emph{Gaussian}.
\end{definition}

A surjective Sch\"urmann triple $(\pi,\eta,\psi)$ is Gaussian if and only if $\pi=\varepsilon\operatorname{id}_D$. See \cite{schuermann90} \cite[Chapter 5]{schuermann93} for equivalent characterizations.

Let us briefly review Sch\"urmann's construction \cite{schuermann90}
\cite[Chapter 5]{schuermann93} that allows to extract the Gaussian part of a cocycle:

Assume that $\eta\colon A\to D$ is surjective (this can always be achieved by
replacing $D$ by $\eta(A)$, if necessary). Let $H=\overline{D}$ be the completion of $D$ and denote by $P_G$, $P_R$ the orthogonal projections onto the closed subspaces
\begin{align*}
H_G &= \big\{\big(\pi(a)-\varepsilon(a)\big)v;a\in A, v\in D\}^\perp, \\
H_R &= \overline{\operatorname{span}\big\{\big(\pi(a)-\varepsilon(a)\big)v;a\in A, v\in D\}},
\end{align*}
of $H$, respectively. We define representations $\pi_G$ and $\pi_R$ on $D_G=P_GD$ and $D_R=P_RD$ by
\begin{align*}
\pi_G(a)P_G\eta(b) &= \varepsilon(a)P_G\eta(b), \\
\pi_R(a)P_R\eta(b) &= \pi(a)\eta(b) - \varepsilon(a)P_G\eta(b)=\bigl(\pi(a)-\varepsilon(a)\bigr)\eta(b)+\varepsilon(a)P_R\eta(b),
\end{align*}
for $a,b\in A$. Then $\eta_G=P_G\circ \eta$ and $\eta_R=P_R\circ \eta$ are
cocycles on $(A,\varepsilon)$ for $\pi_G$ and $\pi_R$, respectively, and $\eta_G$ is furthermore Gaussian.
Note that $D\subseteq D_G\oplus D_R$, $\eta=\eta_G+\eta_R$, and
$\pi_G\oplus\pi_R$ extends $\pi$.
We call a generating functional, a cocycle, or a Sch\"urmann triple
\emph{(purely) non-Gaussian}, if $D_G=\{0\}$. The cocycle $\eta_R$ is non-Gaussian.

Note that a generating functional is both Gaussian and purely non-Gaussian if and only if it is a derivation. Such generating functionals are \emph{trivial} from the stochastic point of view, because they correspond to a deterministic motion. If a $\psi\colon A\to\mathbb{C}$ is a hermitian derivation on an involutive bialgebra, then the associated convolution semigroup of states $\varphi_t=\exp_\star t\psi\colon A\to\mathbb{C}$, $t\ge 0$ consists of $*$-homomomorphisms, which, in the classical case, means that the $\varphi_t$ correspond to Dirac measures.

\begin{definition}\label{def-lk}
We say that a generating functional $\psi$ or a Sch\"urmann triple
$(\pi,\eta,\psi)$ admits a \emph{L\'evy-Khinchin decomposition} if there
exist generating functionals $\psi_G,\psi_R\colon A\to \mathbb{C}$ such that
$(\pi|_{D_G},\eta_G,\psi_G)$ and $(\pi|_{D_R},\eta_R,\psi_R)$ are Sch\"urmann
triples and $\psi=\psi_G+\psi_R$.

The decomposition $\psi=\psi_G+\psi_R$ or
\[
(\pi,\eta,\psi) \subseteq \big(\pi_G\oplus \pi_R,\eta_G+\eta_R,\psi_G+\psi_R\big)
\]
is called a \emph{L\'evy-Khinchin decomposition}.
\end{definition}

Observe that the condition $\psi=\psi_G+\psi_R$ is not crucial in the following sense: If there
exist generating functionals $\psi_G,\psi_R\colon A\to \mathbb{C}$ such that
$(\pi|_{D_G},\eta_G,\psi_G)$ and $(\pi|_{D_R},\eta_R,\psi_R)$ are Sch\"urmann
triples, then automatically $d:=\psi-\psi_G-\psi_R$ is a derivation. When we replace $\psi_G$ by $\psi_G+d$ (or $\psi_R$ by $\psi_R+d$) we get a L\'evy-Khinchin decomposition of $\psi$.

The central topic of this paper is the following question.

\begin{question}
Which generating functionals or Sch\"urmann triples admit a L\'evy-Khinchin decomposition?
\end{question}

We will show later that there exist generating functionals which do not admit
a L\'evy-Khinchin decomposition, cf.\ Proposition \ref{prop-not-lk}.

A positive answer is known for the following cases:
\begin{enumerate}
\item
on commutative involutive bialgebras, cf.\ \cite{schuermann90} \cite[Chapter 5]{schuermann93},
\item
on the ``Brown-Glockner-von Waldenfels bialgebra'' defined by the unitarity
relations, cf.\ \cite{schuermann90} \cite[Chapter 5]{schuermann93},
\item
on the Woronowicz quantum group $SU_q(2)$, cf.\ \cite{skeide94,schuermann+skeide98},
\item
on the compact quantum groups $SU_q(N)$ and $U_q(N)$, cf.\ \cite{franz+kula+lindsay+skeide14},
\item
for generating functionals on involutive Hopf algebras satisfying some
symmetry condition, cf.\ \cite{das+franz+kula+skalski14}.
\end{enumerate}

\subsection{The five properties: (LK), (GC), (NC), (AC), and \texorpdfstring{(H$^2$Z)}{(H2Z)}}\label{sec:four-properties:-lk}

Let $A$ be a $*$-algebra and $\varepsilon\colon A\to\mathbb{C}$ a character on $A$.

\begin{description}
\item[LK]\label{LK}
(\emph{L\'evy-Khinchin} = L\'evy-Khinchin-decomposition property)\\
We say that $(A,\varepsilon)$ has the (LK)-property if for any generating functional $\psi$ there exist generating functionals $\psi_G$ and $\psi_R$ associated to the Gaussian part $\eta_G$ and the ``remainder'' part $\eta_R$ of the cocycle $\eta$ of $\psi$, i.e., all generating functionals admit a L\'evy-Khinchin decomposition.
\item[GC]
(\emph{Gaussians complete} = Gaussian cocycles can be completed to a triple)\\
We say that $(A,\varepsilon)$ has the (GC)-property if any Gaussian cocycle $\eta\colon A\to D$ can be completed to a Sch\"urmann triple $(\varepsilon\operatorname{id}_D,\eta,\psi)$.
\item[NC]
(\emph{Non-Gaussians complete} = Cocycles without Gaussian part can be completed to a triple)\\
We say that $(A,\varepsilon)$ has the (NC)-property if any pair $(\pi,\eta)$ with $\pi$ a unital $*$-representation and $\eta$ a $\pi$-$\varepsilon$-cocycle with $D_G=\{0\}$ can be completed to a Sch\"urmann triple $(\pi,\eta,\psi)$.
\item[AC]
(\emph{All complete} = All cocycles can be completed to a triple)\\
We say that $(A,\varepsilon)$ has the (AC)-property if any pair $(\pi,\eta)$ with $\pi$ a unital $*$-representation and $\eta$ a $\pi$-$\varepsilon$-cocycle can be completed to a Sch\"urmann triple $(\pi,\eta,\psi)$.
\item[H$^2$Z]
(\emph{Second cohomology zero} = the second cohomology with trivial coefficients vanishes) \\
We say that $(A,\varepsilon)$ has the (H$^2$Z)-property if $H^2(A,\mathbb{C})=\{0\}$. See the next Section for the definition of the second cohomology $H^2(A,\mathbb{C})$.
\end{description}

\begin{remark}
In \cite{schuermann90} and \cite[Chapter 5]{schuermann93} the properties (LK), (GC), and (AC) are called (C), (C'), and (D), respectively.
\end{remark}

It is clear that (AC) implies (GC) and (NC), furthermore (GC)$\vee$(NC) implies (LK), see \cite{schuermann90}, \cite[Chapter 5]{schuermann93}. In the next Section we show that (H$^2$Z) implies (AC), see Remark \ref{nt-remark}.

\section{Hochschild (co-)homology}\label{sec-hoch}

Let $M$ be an $A$-bimodule and put $$C^n(A,M):=L(A^{\otimes n},M) = \{\phi\colon A^{\otimes n}\to M; \phi \mbox{ linear}\}.$$ Together with the coboundary operator $\partial\colon C^{n-1}(A,M)\to C^n(A,M)$,
\begin{multline*}
	\partial \phi (a_1\otimes\dots\otimes a_{n}):=
		a_1.\phi(a_2\otimes\dots\otimes a_{n})+\sum_{i=1}^{n-1} (-1)^i~\phi\big(a_1\otimes \dots \otimes (a_ia_{i+1})\otimes \dots\otimes a_{n}\big)\\
        +(-1)^{n}~\phi(a_1\otimes \dots \otimes a_{n-1}).a_{n}
\end{multline*}
this is a cochain complex, it is called the \emph{Hochschild complex} of $M$. The elements of $C^n(A,M)$ are called \emph{($n$-)cochains}. A cochain $\phi$ is called a \emph{cocycle} if $\partial\phi=0$ and a \emph{coboundary} if there exists a cochain $\psi$ with $\phi=\partial\psi$. We denote by $Z^n(A,M)$ the set of all $n$-cocycles, by $B^n(A,M)$ the set of all $n$-coboundaries and by $H^n(A,M):=Z^n(A,M)/B^n(A,M)$ the $n$th cohomology. 

In our context, the bimodule is usually a pre-Hilbert space $D$ with left action given by a unital $*$-representation $\pi$ of $A$ and right action given by the character $\varepsilon$, i.e.,
\[
a.v.b=\pi(a)v\varepsilon(b).
\]
In that case, we speak of $\pi$-$\varepsilon$-cocycles. An important special case is $D=\mathbb{C}$ and $\pi=\varepsilon$, because the generating functionals take values in $\mathbb{C}$.

In this terminology, a derivation is an $\varepsilon$-$\varepsilon$-cocycle.

There is also a notion of Hochschild homology. It will only appear here for the bimodule $\mathbb{C}$ with left and right action implemented by $\varepsilon$. In this case, the chain complex consists of the spaces $C_n=A^{\otimes n}$ which are the pre-duals of $C^n$ as vector spaces and the boundary operator $d\colon C_{n}\to C_{n-1}$ is given by
\begin{multline*}
  d(a_1\otimes\cdots\otimes a_{n})=\varepsilon(a_1)a_2\otimes\cdots
  \otimes a_n+\sum_{i=1}^{n-1}a_1\otimes\cdots\otimes
  a_ia_{i+1}\otimes\cdots\otimes
  a_{n}\\
  +(-1)^{n}~\phi(a_1,\dots,a_{n-1})\varepsilon(a_{n}).
\end{multline*}
Evidently, $\partial$ is the transpose of $d$, i.e., $\partial\phi=\phi\circ d$. The cycles, boundaries and homology groups are denoted by $Z_n(A,\mathbb{C})$, $B_n(A,\mathbb{C})$ and $H_n(A,\mathbb{C})$ respectively.

\begin{proposition}\label{prop:L(eta)-cocycle}
Let $\eta$ be a cocycle on $(A,\varepsilon)$.

The map $\mathcal{L}(\eta)\colon A\otimes A \to \mathbb{C}$,
$\big(\mathcal{L}(\eta)\big)(a\otimes b)=\langle\eta(a^*),\eta(b)\rangle$ is a
$\varepsilon$-$\varepsilon$-$2$-cocycle.
\end{proposition}
\begin{proof}
This is a special case of the so-called cup-product, a direct proof is as follows:
\begin{align*}
\partial\bigl(\mathcal{L}&(\eta)\bigr)(a\otimes b\otimes c) \\
&=\varepsilon(a) \langle\eta(b^*),\eta(c)\rangle -
\bigl\langle\eta\bigl((ab)^*\bigr),\eta(c)\bigr\rangle +
\langle\eta(a^*),\eta(bc)\rangle -
\langle\eta(a^*),\eta(b)\rangle\varepsilon(c) \\
&=\varepsilon(a) \langle\eta(b^*),\eta(c)\rangle -
\bigl\langle\bigl(\pi(b)^*\eta(a^*)+ \eta(b^*)\overline{\varepsilon(a)}
\bigr),\eta(c)\bigr\rangle \\
&\quad+ \bigl\langle\eta(a^*),\bigl(\pi(b)\eta(c)
+ \eta(b)\varepsilon(c)\bigr)\bigr\rangle -
\langle\eta(a^*),\eta(b)\rangle\varepsilon(c) \\
&= 0
\end{align*}
for all $a,b,c\in A$.
\end{proof}

If $\eta\colon A\to D$ is a coboundary, i.e., if there exists a vector $v\in D$
such that
\[
\eta(a) = \big(\pi(a)-\varepsilon(a)\big) v \qquad \mbox{ for } a\in A,
\]
then $\mathcal{L}(\eta)$ is also a coboundary. In that case we have
\[
\mathcal{L}(\eta) = - \partial \phi_v
\]
with
\[
\phi_v(a) = \big\langle v, \big(\pi(a)-\varepsilon(a)\big) v\big\rangle \qquad
\mbox{ for } a\in A.
\]

The following theorem gives a new answer to the question of existence and
uniqueness of a generating functional to a given cocycle. 
\begin{theorem}\label{nt-sequence}
We have the exact sequence
\[
0 \to H_2(A,\mathbb{C}) \to K_1\otimes_A K_1 \stackrel{\mu}{\to} K_1 \to H_1(A,\mathbb{C}) \to 0
\]
\end{theorem}
\begin{proof}
This result is stated for group algebras in \cite[Lemma 5.6]{netzer+thom13}. The proof does not use any group properties, so it is
clear that it extends to general algebras.

It can also be verified by direct calculation. The map from
$H_2(A,\mathbb{C})$ to $K_1\otimes_A K_1$ is induced by the map from
$Z_2(A,\mathbb{C})$ to $K_1\otimes K_1$ given by the tensor product of the
projection from $A$ to $K_1$, $a\mapsto a-\varepsilon(a)\mathbf{1}$, with
itself, and the canonical projection from $K_1\otimes K_1$ to $K_1\otimes_A
K_1$. I.e., we have the map
\[
\hat{p}\colon Z_2(A,\mathbb{C}) \ni a\otimes b \mapsto
\big(a-\varepsilon(a)\mathbf{1}\big)\otimes\big(b-\varepsilon(b)\mathbf{1}\big)\in
K_1\otimes_A K_1.
\]
It is straight-forward to check that $\hat{p}$ vanishes on $B_2(A,\mathbb{C})$
and induces an injective map $p\colon H_2(A,\mathbb{C})\to K_1\otimes_A K_1$.

The map from $\mu\colon K_1\otimes_A K_1 \to K_1$ is multiplication, $\mu(a\otimes
b)=ab$, its kernel is the image $p\big(H_2(A,\mathbb{C})\big)$ and its range is the ideal
\[
K_2={\rm span}\{ab;a,b\in K_1\}.
\]
Exactness in $H_1(A,\mathbb{C})$ follows from $H_1(A,\mathbb{C})\cong K_2/K_1$.
\end{proof}

The cohomological version of this result holds as well:

\begin{theorem}\label{Theorem:exact-sequence-cohomology}
  We have the exact sequence
\[
0 \to H^1(A,\mathbb{C})\to K_1' \stackrel{\mu'}{\to}  (K_1\otimes_A K_1)'\to H^2(A,\mathbb{C})  \to 0
\]
\end{theorem}

\begin{proof}
  Since $B^1(A,\mathbb{C})=0$, $H^1(A,\mathbb{C})\cong Z^1(A,\mathbb{C})$ is the space of all derivations on $A$. We can define $R\colon H^1(A,\mathbb{C})\to K_1', R([\phi]):=\phi|_{K_1}$. For a derivation $\phi$ we have $0=\partial\phi(1\otimes1)=\phi(1)$. So $\phi|_{K_1}=0$ together with $\phi\in Z^1$ implies $\phi=0$, hence $R$ is injective. A linear functional $\varphi\colon K_1\to\mathbb{C}$ extends to a derivation on $A$ if and only if $\varphi(ab)=0$ for all $a,b\in K_1$. This shows exactness at $K'_1$. An element $S\in(K_1\otimes_A K_1)'$ lifts to a linear map $\hat S\colon K_1\otimes K_1\to \mathbb{C}$ with $\hat{S}(ab\otimes c)=\hat{S}(a\otimes bc)$ for all $a,b,c\in K_1$. Extending $\hat{S}$ to $A\otimes A$ by $\hat S(a,b):=\hat{S}((a-\varepsilon(a))\otimes (b-\varepsilon(b)))$ yields a 2-cocycle. Define the  linear map $L\colon (K_1\otimes_A K_1)'\to H^2(A,\mathbb{C})$ with $L(S)=[\hat S]$. Then $L(S)=0$ if and only if $\hat S=\partial \psi$ for some $\psi\colon A\to\mathbb{C}$. The linear functional $\psi$ can always be chosen such that $\psi(1)=0$. In that case $\hat S=\partial \psi$ is equivalent to $S=\psi|_{K_1}\circ\mu$. So we have exactness at $(K_1\otimes_A K_1)'$. Finally, every $T\in Z^2$ fulfills $T(ab\otimes c)=T(a\otimes bc)$ for all $a,b,c\in K_1$, so its restriction to $K_1\otimes K_1$ descends to a linear functional $\widetilde{T}$ on the quotient space $K_1\otimes_A K_1$. We can subtract the coboundary $T(1\otimes 1)\partial{\varepsilon}=T(1\otimes 1)\varepsilon\otimes\varepsilon$ from $T$ to get a new cocycle $T_0$ with $[T]=[T_0]$ and $T_0(1\otimes b)=T(a\otimes 1)=0$ for all $a,b\in A$. Now it is easy to check that $L(\widetilde{T_0})=[T_0]=[T]$, so $L$ is surjective.
\end{proof}

\begin{lemma}\label{nt2-lemma}
  For any $\varepsilon$-$\varepsilon$-$2$-cocycle $\phi\colon A\otimes A\to \mathbb{C}$ there exists a unique linear map $\varphi\colon K_1\otimes_A K_1\to \mathbb{C}$ such that
  \begin{align*}
    \varphi(a\otimes b)=\phi(a\otimes b)
  \end{align*}
  for all $a,b\in K_1$.
\end{lemma}

\begin{proof}
  The restriction of $\phi$ to $K_1\otimes K_1$ passes to the quotient $K_1\otimes_A K_1$, because
  \begin{align*}
    0=\partial \phi(a\otimes b\otimes c)=-\phi(ab\otimes c)+\phi(a\otimes bc)
  \end{align*}
  for all $a,b,c\in K_1$.
\end{proof}

\begin{corollary}\label{nt-lemma}
For any cocycle $\eta\colon A\to H$ there exists a unique linear map $\mathcal{K}(\eta)\colon K_1\otimes_A K_1\to \mathbb{C}$ such that
\begin{equation}\label{eq-def-phi}
\big(\mathcal{K}(\eta)\big)(a\otimes b) = \langle\eta(a^*),\eta(b)\rangle
\end{equation}
for all $a,b\in K_1$, 
\end{corollary}
\begin{proof}
By Proposition \ref{prop:L(eta)-cocycle}, $\mathcal{L}(\eta)$ is a cocycle, hence we can apply Lemma \ref{nt2-lemma}. 
\end{proof}

\begin{remark}\label{nt-remark}
Combining Lemma \ref{nt2-lemma} and the exact sequence
\[
0 \to H_2(A,\mathbb{C})\to K_1\otimes_A K_1 \to K_1 \to H_1(A,\mathbb{C})\to 0
\]
from Theorem \ref{nt-sequence}, we see that for a given functional $\varphi\colon K_1\otimes_A K_1\to\mathbb{C}$ there exists a functional $\psi\colon A\to \mathbb{C}$ such that $\psi(\mathbf{1})=0$ and
\[
\psi(ab) = \varphi(a\otimes b) \qquad \mbox{ for } a,b\in K_1
\]
exists if and only if $\varphi$ vanishes on the range of the map from $H_2(A,\mathbb{C})$ to $K_1\otimes_A K_1$. By exactness, the range of this latter map coincides with the kernel of the map $\mu\colon K_1\otimes_A K_1\to K_1$. And $\psi$ is determined by $\varphi$ up to a linear functional on $H_1(A,\mathbb{C})$. It follows that $H_2(A,\mathbb{C})=0$ if and only if $H^2(A,\mathbb{C})=0$. For $\phi=\mathcal{L}(\eta)$, $\psi$ can alway be chosen hermitian, so it follows that a cocycle $\eta$ admits a generating functional if and only if $\mathcal{K}(\eta)$ vanishes on ${\rm ker}(\mu)\cong H_2(A,\mathbb{C})$.

We will abbreviate the condition $H^2(A,\mathbb{C})=\{0\}$ as (H$^2$Z), the discussion above shows that this condition implies the property (AC). In Subsection \ref{exa-not2=0} we shall prove that (H$^2$Z) is strictly stronger than (AC).
\end{remark}

\section{Examples}\label{sec-examples}

In this section we prove the existence of generating functionals that do not admit a L\'evy-Khinchin decomposition, see Proposition \ref{prop-not-lk}. We also study explicit examples that show that none of the other converse implications in Diagram \eqref{diagram} holds. Here is an overview of where the relevant counter-examples can be found:

\begin{center}
\begin{tabular}{rl}
(AC) $\centernot\implies$ (H$^2$Z)~: & \ref{exa-not2=0}\\
(NC) $\centernot\implies$ (AC)~: & \ref{exa-abelian} \\
(GC) $\centernot\implies$ (AC)~: & \ref{exa-p2} \\
(GC) $\centernot\implies$ (NC)~: & \ref{exa-p2} \\
(NC) $\centernot\implies$ (GC)~: & \ref{exa-abelian} \\
(LK) $\centernot\implies$ (GC)$\vee$(NC)~: & \ref{exa-free-pr}
\end{tabular}
\end{center}

If we want to define cocycles or Sch\"urmann triples on an algebra defined by
generators and relations, then it is enough to choose the values on the
generators and check that all maps vanish on the relations.

\subsection{The fundamental groups of a closed oriented surface of
  genus \texorpdfstring{$k\ge 2$}{k > 1}}\label{ex-surface-group}

The fundamental group $\Gamma_k$ of an oriented surface of
  genus $k\ge 1$ has a presentation
\[
\Gamma_k = \langle a_i,\ldots,a_k,b_1,\ldots,b_k|
a_1b_1a_1^{-1}b_1^{-1}a_2b_2a_2^{-1}b_2^{-1}\cdots
a_kb_ka_k^{-1}b_k^{-1}\rangle
\]
with $2k$ generators $a_1,\ldots, a_k,b_1,\ldots,b_k$. For $k=1$ we have $\Gamma_1\cong \mathbb{Z}^2$, this case will be treated in Subsection \ref{exa-abelian}. Let now $k\ge 2$. We will consider the group $*$-algebra $A=\mathbb{C}\Gamma_k$ with the character given by the trivial representation, i.e., $\varepsilon(a_\ell)=\varepsilon(b_\ell)=1$ for $\ell=1,\ldots, k$. Then we have $H_1(\mathbb{C}\Gamma_k,\mathbb{C})=\mathbb{C}^{2k}$ and $H_2(\mathbb{C}\Gamma_k,\mathbb{C})=\mathbb{C}$, cf.\ \cite[II.4 Example 2]{brown82}.

\begin{proposition}
Let $D$ be a pre-Hilbert space.
\begin{enumerate}[label=\textnormal{(\alph*)},leftmargin=*]
\item For any $x_1,\ldots, x_k,y_1,\ldots,y_k\in D$ there
  exists a unique Gaussian cocycle on
  $(\mathbb{C}\Gamma_k,\varepsilon)$ with
  \begin{align*}
    \eta(a_j) = x_j,\quad \eta(b_j) = y_j,
  \end{align*}
  for $j=1,\ldots,k$.
\item A Gaussian cocycle on $(\mathbb{C}\Gamma_k,\varepsilon)$
  admits a generating functional if and only if
  \[
  \sum_{j=1}^k \langle \eta(a_j),\eta(b_j)\rangle \in\mathbb{R}.
  \]
  In particular, $(\mathbb{C}\Gamma_k,\varepsilon)$ does not have
  \textnormal{(GC)}.
\end{enumerate}

\end{proposition}
\begin{proof}\ 
\begin{enumerate}[label=\textnormal{(\alph*)},leftmargin=*]
\item
For any $x_1,\ldots, x_k,y_1,\ldots,y_k\in D$ we can define a Gaussian cocycle $\tilde{\eta}$ on the free group $\mathbb{F}_{2k}$ (or its group $*$-algebra $\mathbb{C}\mathbb{F}_{2k}$). Because of the universality of the free group, we can do this simply by defining $\tilde{\eta}(a_\ell)=x_\ell$, $\eta(b_\ell)=y_\ell$ for $\ell=1,\ldots,k$ --- we denote the $2k$ generators of $\mathbb{F}_{2k}$ also by $a_1,\ldots,a_k,b_1,\ldots,b_k$ --- and extending $\tilde{\eta}$ as a derivation.  To get a cocycle on $\mathbb{C}\Gamma_k$ we have to check that this cocycle respects the defining relation of $\Gamma_k$, i.e., that $\tilde{\eta}(a_1b_1a_1^{-1}b_1^{-1}\cdots a_kb_ka_k^{-1}b_k^{-1})=\tilde{\eta}(\mathbf{1})=0$. Since Gaussian cocycles are derivations, we have $\tilde{\eta}(g^{-1})=-\tilde{\eta}(g)$ for all $g\in\mathbb{F}_{2k}$ and
\begin{align*}
\tilde{\eta}(a_1b_1a_1^{-1}&b_1^{-1}\cdots a_kb_ka_k^{-1}b_k^{-1}) \\&= \tilde{\eta}(a_1) + \tilde{\eta}(b_1)-\tilde{\eta}(a_1) - \tilde{\eta}(b_1) + \cdots - \tilde{\eta}(b_k) = 0.
\end{align*}
\item
The free group has $H^1(\mathbb{C}\mathbb{F}_{2k},\mathbb{C})=\mathbb{C}^{2k}$and $H^2(\mathbb{C}\mathbb{F}_{2k},\mathbb{C})=\{0\}$, cf.\ \cite[II.4 Example 1]{brown82} or \cite[Corollaire I.6.1]{guichardet80}, therefore $(\mathbb{C}\mathbb{F}_{2k},\mathbb{C})$ has the (AC)-property. Let $\eta$ be a Gaussian cocycle on $(\mathbb{C}\Gamma_k,\varepsilon)$. Denote by $\tilde{\eta}$ the cocycle on $(\mathbb{C}\mathbb{F}_{2k},\mathbb{C})$ obtained by composing $\eta$ with the canonical projection $\mathbb{C}\mathbb{F}_{2k}\to\mathbb{C}\Gamma_k$ and let $\tilde{\psi}$ be a generating functional for $\tilde{\eta}$. Then $\eta$ admits a generating functional iff we can choose $\tilde{\psi}$ such that it vanishes on the ideal generated by $a_1b_1a_1^{-1}b_1^{-1}\cdots a_kb_ka_k^{-1}b_k^{-1}-\mathbf{1}$. This is the case iff
\begin{align}
0 &= \tilde{\psi}(\mathbf{1}) = \tilde{\psi}(a_1b_1a_1^{-1}b_1^{-1}\cdots a_kb_ka_k^{-1}b_k^{-1}) \nonumber \\
&=
\tilde{\psi}(a_1) + \langle\underbrace{\tilde{\eta}(a_1^{-1})}_{ -\tilde{\eta}(a_1)},\underbrace{\tilde{\eta}(b_1a_1^{-1}b_1^{-1}\cdots a_kb_ka_k^{-1}b_k^{-1})}_{-\tilde{\eta}(a_1)}\rangle\nonumber \\ &\quad+  \tilde{\psi}(b_1a_1^{-1}b_1^{-1}\cdots a_kb_ka_k^{-1}b_k^{-1}) \nonumber \\
&= \tilde{\psi}(a_1) + \langle\tilde{\eta}(a_1),\tilde{\eta}(a_1)\rangle \nonumber \\
&\quad + \tilde{\psi}(b_1) + \langle\underbrace{\tilde{\eta}(b_1^{-1})}_{-\tilde{\eta}(b_1)},\underbrace{\tilde{\eta}(a_1^{-1}b_1^{-1}\cdots a_kb_ka_k^{-1}b_k^{-1})}_{-\tilde{\eta}(a_1)-\tilde{\eta}(b_1)}\rangle  \nonumber \\ &\quad+  \tilde{\psi}(a_1^{-1}b_1^{-1}\cdots a_kb_ka_k^{-1}b_k^{-1}) \nonumber \\
&=  \tilde{\psi}(a_1) + \langle\tilde{\eta}(a_1),\tilde{\eta}(a_1)\rangle + \tilde{\psi}(b_1) + \langle\tilde{\eta}(b_1),\tilde{\eta}(b_1)\rangle + \langle\tilde{\eta}(b_1),\tilde{\eta}(a_1)\rangle \nonumber \\
& \quad+ \tilde{\psi}(a_1^{-1}) +  \langle\tilde{\eta}(a_1),\underbrace{\tilde{\eta}(b_1^{-1}\cdots a_kb_ka_k^{-1}b_k^{-1})}_{-\tilde{\eta}(b_1)}\rangle  +  \langle\tilde{\psi}(b_1^{-1}\cdots a_kb_ka_k^{-1}b_k^{-1}) \nonumber \\
&= \cdots \nonumber \displaybreak[0]\\
&= 
\sum_{\ell=1}^k \left(\tilde{\psi}(a_\ell) + \tilde{\psi}(a^{-1}_\ell) + \langle\tilde{\eta}(a_\ell),\tilde{\eta}(a_\ell)\rangle\right) \nonumber \\ &\quad+ \sum_{\ell=1}^k \left(\tilde{\psi}(b_\ell) + \tilde{\psi}(b^{-1}_\ell) + \langle\tilde{\eta}(b_\ell),\tilde{\eta}(b_\ell)\rangle\right) \nonumber \\
& \quad+ \sum_{\ell=1}^k \left(\langle\tilde{\eta}(b_\ell),\tilde{\eta}(a_\ell)\rangle - \langle\tilde{\eta}(a_\ell),\tilde{\eta}(b_\ell)\rangle\right),
\label{eq-cond}
\end{align}
where we used repeatedly the fact that $-\partial \tilde{\psi}$ is equal to $\mathcal{L}(\tilde{\eta})$, i.e., $\psi(g_1g_2)=\psi(g_1) + \langle\tilde{\eta}(g_1^{-1}),\tilde{\eta}(g_2)\rangle+\psi(g_2)$ for $g_1,g_2\in\mathbb{F}_{2k}$. But $-\partial \tilde{\psi}=\mathcal{L}(\tilde{\eta})$ implies also
\begin{eqnarray*}
0&=&\tilde{\psi}(\mathbf{1}) = \tilde{\psi}(g^{-1}g) \\
&=& \tilde{\psi}(g) + \tilde{\psi}(g^{-1}) + \langle\tilde{\eta}(g),\tilde{\eta}(g)\rangle
\end{eqnarray*}
for any $g\in\mathbb{F}_{2k}$,
and therefore the first two sums in the final expression in Equation \eqref{eq-cond} vanish. The remaining third sum is equal to
\[
\sum_{\ell=1}^k {\rm Im}\left(\langle\tilde{\eta}(b_\ell),\tilde{\eta}(a_\ell)\rangle\right).
\]
which leads to the desired condition for the existence of $\psi$.\qedhere
\end{enumerate}
\end{proof}

\begin{proposition}
$(\mathbb{C}\Gamma_k,\varepsilon)$ does not have \textnormal{(NC)} for $k\ge 2$.
\end{proposition}
\begin{proof}
It is sufficient to prove this for $k=2$. We consider the $*$-represention given by
\begin{gather*}
\pi(a_1)=\pi(b_1)=\pi(a_2) = {\rm id}_D, \\
\pi(b_2) = -{\rm id}_D
\end{gather*}
on some pre-Hilbert space $D$.
There exists a cocycle $\eta\colon\mathbb{C}\Gamma_2\to D$ with
\begin{align*}
\eta(a_j) &= x_j, \\
\eta(b_j) &= y_j, 
\end{align*}
$j=1,2$, if and only if
\begin{align*}
\eta(a_1b_1a_1^{-1}&b_1^{-1}a_2b_2a_2^{-1}b_2^{-1}) \\
&= -\eta(b_2^{-1}) -
\eta(a_2^{-1})+\eta(b_2)+\eta(a_2) + \eta(b_1^{-1}) +\eta(a_1^{-1})
+\eta(b_1)+\eta(a_1) \\
&= y_2-x_2+y_2+x_2-y_1-x_1+y_1+x_1 \\
&= 2y_2 \stackrel{\mbox{!}}{=} 0.
\end{align*}
By Remark \ref{nt-remark}, there exists a generating functional for such a cocycle if and only if $\mathcal{L}(\eta)$ vanishes on
\begin{align*}
c_2 &= (a_1^{-1}-1)\otimes(b_1^{-1}-1)b_2a_2 - (b_1^{-1}-1)\otimes
(a_1^{-1}-1)b_2a_2 \\
&\quad + a_1^{-1}b_1^{-1}(a_2-1)\otimes(b_2-1)-a_1^{-1}b_1^{-1}(b_2-1)\otimes(a_2-1),
\end{align*}
since $c_2$ spans ${\rm ker}(\mu)\cong H_2(\mathbb{C}\Gamma_2,\mathbb{C})$.
After some calculation one finds that this is equivalent to the condition
\[
\langle x_1,y_1\rangle - \langle y_1,x_1\rangle \stackrel{\mbox{!}}{=} 
\langle x_2,y_2\rangle - \langle y_2-2y_1-2x_1,x_2\rangle.
\]
Take, e.g., $D=\mathbb{C}$, $x_1=x_2=1$, $y_1=y_2=0$, then there exists a
cocycle with $\eta(a_i)=x_i$, $\eta(b_i)=y_i$ for $i=1,2$. But there exists no generating functional for this cocycle.
\end{proof}

\begin{proposition}\label{prop-not-lk}
$(\mathbb{C}\Gamma_k,\varepsilon)$ does not have ${\rm (LK)}$ for $k\ge 2$.
\end{proposition}
\begin{proof}
Take a direct sum $\eta=\eta_1\oplus \eta_2$ of a Gaussian cocycle $\eta_1$ and a non-Gaussian cocycle $\eta_2$, which admit no
generating functionals, in such a way that $\mathcal{L}(\eta)=\mathcal{L}(\eta_1)+\mathcal{L}(\eta_2)$ vanishes on $c_2$. Then this direct sum does admit a generating functional, but the resulting generating functional does not admit a L\'evy-Khinchin decomposition. This is possible, because we can choose the values of the Gaussian cocycle $\eta_1$ on the generators $a_1,a_2,b_1,b_2$ such that
\[
\mathcal{L}(\eta_1)(c_2) = {\rm Im}\left(\langle \eta(b_1),\eta(a_1)\rangle + \langle \eta(b_2),\eta(a_2)\rangle\right)
\]
takes any real number we want as value.
\end{proof}

\subsection{Free abelian groups}\label{exa-abelian}

For the free abelian groups $\mathbb{Z}^k$, $k\ge 1$, and the character $\varepsilon\colon\mathbb{C}\mathbb{Z}^k\to\mathbb{C}$ coming from the trivial representation, we have $H_1(\mathbb{C}\mathbb{Z}^k,\mathbb{C})=\mathbb{C}^k$ and $H_2(\mathbb{C}\mathbb{Z}^k,\mathbb{C})=\mathbb{C}^{\frac{k(k-1)}{2}}$, cf.\ \cite[II.4 Example 4]{brown82} or \cite[Proposition I.6.2]{guichardet80}. \cite[Theorem 3.12]{schuermann90} by Sch\"urmann implies that $(\mathbb{C}\mathbb{Z}^k,\varepsilon)$ has property (LK). Actually it also has property (NC).

\begin{proposition}
Any purely non-Gaussian cocycle on $\mathbb{C}\mathbb{Z}^k$ admits a generating functional.
\end{proposition}
\begin{proof}
This results holds actually for all discrete abelian groups, it can be deduced from a result by Skeide \cite{skeide99}. The $*$-algebra $\mathbb{C}\Gamma$ of a discrete abelian group $\Gamma$ is isomorphic to the $*$-algebra $\mathcal{R}(\hat{\Gamma})$ generated by the coefficients of a faithful finite-dimensional representation of its dual group $\hat{\Gamma}$. In \cite[Section 3.2]{skeide99} it is shown that any purely non-Gaussian cocycle on the $*$-Hopf algebra $\mathcal{R}(\hat{\Gamma})$ of representative functions on a compact group $\hat{\Gamma}$ admits a generating functional, see in particular \cite[Equation (6)]{skeide99}.
\end{proof}

But for $k\ge 2$,  $(\mathbb{C}\mathbb{Z}^k,\varepsilon)$ does not have property (GC). It is sufficient to consider $k=2$. Since ${\rm ker}(\mu)\cong H_2(\mathbb{C}\mathbb{Z}_2,\mathbb{C})$ is spanned by
\[
c_1 = (a^{-1}-1)\otimes(b^{-1}-1) - (b^{-1}-1)\otimes
(a^{-1}-1),
\]
where $a$ and $b$ denote the two canonical generators of $\mathbb{Z}^2$, we can show that a Gaussian cocycle on $(\mathbb{C}\mathbb{Z}^k,\varepsilon)$ has a generating functional if and only if
\[
\langle\eta(a),\eta(b)\rangle \in \mathbb{R}.
\]
Therefore $(\mathbb{C}\mathbb{Z}^k,\varepsilon)$ does not have the properties (GC) or (AC) for $k\ge 2$.

\subsection{The wallpaper group \texorpdfstring{``$p2$''}{''p2''}}\label{exa-p2}

Let $G$ be the wallpaper group ``$p2$'', i.e., the subgroup of
Isom($\mathbb{R}^2$) generated by two translations $a$ and $b$ (in two
linearly independent directions) and a rotation $r$ by $180^\circ$. This group
has a presentation
\[
G = \langle a,b,r | aba^{-1}b^{-1}=r^2=(ra)^2=(rb)^2=1\rangle.
\]

\begin{proposition}
There are no non-zero Gaussian cocycles on
$(\mathbb{C}G,\varepsilon)$. Therefore $(\mathbb{C}G,\varepsilon)$ has the
properties \textnormal{(GC)} and \textnormal{(LK)}.
\end{proposition}
\begin{proof}
Recall that Gaussian cocycles are simply
($\varepsilon$-$\varepsilon$-)derivations. Since we can view $G$ as generated
by the three elements $r$, $ra$, and $rb$, which have order two, there exist no
non-zero derivations on $(\mathbb{C}G,\varepsilon)$.
\end{proof}

\begin{proposition}
$(\mathbb{C}G,\varepsilon)$ has non-Gaussian cocycles which do not admit a
generating functional. Therefore $(\mathbb{C}G,\varepsilon)$ does not have the
properties \textnormal{(NC)} or \textnormal{(AC)}.
\end{proposition}
\begin{proof}
We consider the representation given by
\begin{gather*}
\pi(a)=\pi(b) = {\rm id}_D, \\
\pi(r)=-{\rm id}_D,
\end{gather*}
on a pre-Hilbert space $D$. Then there exists a unique cocycle $\eta\colon A\to D$
with
\begin{align*}
\eta(a) &= x, \\
\eta(b) &= y, \\
\eta(r) &= z,
\end{align*}
for any $x,y,z\in D$,
since
\begin{align*}
\eta(aba^{-1}b^{-1}) &= \pi(aba^{-1})\eta(b^{-1}) + \pi(ab)\eta(a^{-1}) +
\pi(a)\eta(b) +  \eta(a) \\
&= -\eta(b)-\eta(a)+\eta(b) +\eta(a) \\
&= 0, \\
\pi(r^2) &= \pi(r) \eta(r) + \eta(r) = -\eta(r)+\eta(r) = 0, \\
\pi\big((ra)^2\big) &= \pi(rar)\eta(a) + \pi(ra)\eta(r) + \pi(r)\eta(a) +
\eta(r) \\
&= \eta(a)-\eta(r) - \eta(a)+\eta(r) \\
&= 0, \\
\pi\big((rb)^2\big) &= \pi(rbr)\eta(b) + \pi(rb)\eta(r) + \pi(r)\eta(b) +
\eta(r) \\
&= \eta(b)-\eta(r) - \eta(b)+\eta(r) \\
&= 0.
\end{align*}
But such a cocycle can only admit a generating functional if
\begin{align*}
-\langle x,y\rangle &= \langle\eta(a^*),\eta(b)\rangle \\
&= L(ab)-\overline{L(a)}-L(b) \\
&= L(ba)-\overline{L(a)}-L(b) \\
&= \langle\eta(b^*),\eta(a)\rangle \\
&= -\langle y, x\rangle,
\end{align*}
i.e., if $\langle x,y\rangle\in\mathbb{R}$.
\end{proof}

\subsection{The free product of \texorpdfstring{$\mathbb{Z}^k$ with ``$p2$''}{Z with ''p2''}}\label{exa-free-pr}

Let $G$ now be the free product of the wallpaper group ``$p2$'' with $\mathbb{Z}^k$, $k\ge 2$, and consider the character $\varepsilon\colon\mathbb{C}G\to\mathbb{C}$ obtained by linear extension of the trivial representation.

\begin{proposition}
$(\mathbb{C}G,\varepsilon)$ does not have properties {\rm (GC)} or {\rm (NC)}, but it has property {\rm (LK)}.
\end{proposition}
\begin{proof}
This is clear because $*$-representations, cocycles, and Sch\"urmann triples on $(\mathbb{C}G,\varepsilon)$ are uniquely determined by their restrictions to the group $*$-algebras of ``$p2$'' and $\mathbb{Z}^k$.
\end{proof}

\subsection{An example to show  \texorpdfstring{$\textnormal{(AC)} \centernot\implies {\rm (H^2Z)}$}{(AC) does not imply (H2Z)} }\label{exa-not2=0}

Consider the unital $*$-algebra 
\[A\colon=\mathbb{C}\langle x,x^*,y,y^*| x^*=x,x^2y=-y,y^*y=0\rangle\]
with the character $\varepsilon$ given on the generators by $\varepsilon(x)=\varepsilon(y)=0$.

We want to show that \textnormal{(AC)} holds.
\begin{proposition}
  Let $\pi$ be a $*$-representation of $A$ on a pre-Hilbert space $D$ and $\eta\colon{A}\to D$ a $\pi$-$\varepsilon$-$1$-cocycle. Then $\pi(y)=\pi(y^*)=0$ and $\eta(y)=\eta(y^*)=0$.
\end{proposition}
\begin{proof}
The third relation yields $0=\pi(y^*y)=\pi(y)^*\pi(y)$, which implies $\pi(y)=\pi(y^*)=0$. From the first and second relation we get 
\[
0\leq\langle\eta(y),\eta(y)\rangle=\langle{\eta(y),\eta(-x^2y)}\rangle=
-\langle\pi(x)\eta(y),\pi(x)\eta(y)\rangle\le0,
\]
which implies
$\eta(y)=0$, and $\eta(y^*)=-\eta(y^*x^2)=\pi(y^*)\eta(x^2)=0$.
\end{proof}
From this proposition and the cocycle identitity we conclude, with $\pi(x)=:A$ and $\eta(x)=:v$, that $\eta(x^k)=A^kv$ and $\eta(M)=0$ for every monomial which contains either $y$ or $y^*$. We define 
\[\psi(M):=
\begin{cases}
  -\langle{v,A^{k}v}\rangle&\text{for $M=x^{k+2}, k\in\mathbb{N}$}\\
  0&\text{otherwise}
\end{cases}
\] 
(note that the relations, except $x^*=x$, all involve $y$, so they are clearly respected).
Then we obviously have $\langle\eta(M^*),\eta(N)\rangle=-\psi(MN)$ for all monomials with $\varepsilon(M)=\varepsilon(N)=0$. But, since $\eta(1)=0$, $\psi(1)=0$ and $\pi(1)=\operatorname{id}$, that is enough to have $\langle\eta(a^*),\eta(b)\rangle=\partial\psi(ab)$ for all $a,b\in{A}$. Thus, we have shown that if $\eta$ is a $\pi$-$\varepsilon$-$1$-cocycle for a $*$-representation $\pi$ on a pre-Hilbert space, then $\mathcal{L}(\eta)$ ia a coboundary, so (AC) holds.

Next we give a nontrivial 2-cocycle, which shows that $H^2\neq\{0\}$.
On the two-dimensional complex vector space $\mathbb{C}^2$ we define the non-degenerate sesquilinear hermitian form $\langle\cdot,\cdot\rangle\to\mathbb{C}$, given by the matrix \[J:=
\begin{pmatrix}
  1&0\\0&-1
\end{pmatrix},
\]
i.e., $
\langle v,w\rangle = \overline{v_1}w_1-\overline{v_2}w_2=\overline{v^t}Jw$
for $v=(v_1,v_2)^t, w=(w_1,w_2)^t\in\mathbb{C}^2$. Every linear map from $\mathbb{C}^2$ to itself is adjointable, and if $A$ is its representing $2\times 2$-matrix, then $A^\dagger:=JA^*J$ represents its adjoint. Together with the involution $\dagger$, the matrix algebra $M_2(\mathbb{C})$ becomes a unital $*$-algebra. We define a unital $*$-representation $\pi\colon A\to (M_2(\mathbb{C}),\dagger)$ on $(\mathbb{C}^2,\langle\cdot,\cdot\rangle)$ and a $\pi$-$\varepsilon$-cocycle $\eta\colon A\to\mathbb{C}^2$ by assigning
\begin{align*}
\pi(x)&=\left(\begin{array}{cc} 0 & 1 \\ -1 & 0\end{array}\right) &\pi(y)&=0&
\eta(y)&=
         \begin{pmatrix}
           1\\0
         \end{pmatrix}
&\eta(x)&=\eta(y^*)=0
\end{align*}
to the generators. The corresponding $\varepsilon$-$\varepsilon$-2-cocycle
\[
c(a\otimes b) = \mathcal{L}(\eta)(a\otimes b) = \langle \eta(a^*),\eta(b)\rangle
\]
for $a,b\in{A}$ is nontrivial: The exact sequence of Theorem~\ref{Theorem:exact-sequence-cohomology} tells us that $[c]\in H^2(A,\varepsilon)$ is the image of the corresponding linear functional $\widetilde{c}\in (K_1\otimes_A K_1)'$. Because $c(y^*\otimes y)=\langle \eta(y),\eta(y) \rangle= 1$ and $\mu(y^*\otimes y)=0$, we conclude that $\widetilde{c}\notin\operatorname{im}\mu'$. By exactness it follows that $[c]\neq 0$.

\begin{remark}
This is the only counter-example for which we could not find a group algebra. We do not know if ${\rm (H^2Z)}$ and (AC) might be equivalent under reasonable additional assumptions that are verified by group algebras, such as the existence of a faithful state.
\end{remark}

\par\bigskip\noindent
{\bf Acknowledgment.}
We acknowledge support by MAEDI/MENESR and DAAD through the PROCOPE programme.

\bibliographystyle{amsplain}

\begin{thebibliography}{99}

\bibitem{brown82}
Brown, K. S.:
\newblock {\em Cohomology of groups}, volume~87 of {\em Graduate Texts in
  Mathematics},
\newblock Springer-Verlag, New York-Berlin, 1982.

\bibitem{das+franz+kula+skalski14}
Das, B., Franz, U., Kula, A., Skalski, A.:
One-to-one correspondence betwen generating functionals and cocycles on quantum groups in the presence of symmetry, {\em Math. Z.} {\bf 281} (2015) 949--965.

\bibitem{franz+kula+lindsay+skeide14}
Franz, U., Kula, A., Lindsay, M.,  Skeide M.:
Hunt's formula for $SU_q(N)$ and $U_q(N)$, in preparation, 2015.

\bibitem{franz+schott99}
Franz, U.,  Schott, R.:
{\em Stochastic Processes and Operator Calculus on Quantum Groups}, Kluwer Academic Publishers, Dordrecht, 1999.

\bibitem{guichardet72a}
Guichardet, A.:
{\em Symmetric {H}ilbert spaces and related topics}, volume 261 of {\em Lecture Notes in Math.},
Springer-Verlag, Berlin, 1972.

\bibitem{guichardet80}
Guichardet, A.:
{\em Cohomologie des groupes topologiques et des alg\`ebres de {L}ie}, Textes Math\'ematiques, Vol.\ 2, CEDIC, Paris, 1980.

\bibitem{hunt56}
Hunt, G. A.:.
Semi-groups of measures on {L}ie groups,
{\em Trans. Amer. Math. Soc.} {\bf 81} (1956), 264--293.

\bibitem{netzer+thom13}
Netzer, T., Thom, A.: Real closed separation theorems and applications to group algebras, {\em Pacific J. Math.} {\bf 263} (2013) 435--452.

\bibitem{parthasarathy+schmidt72}
Parthasarathy, K.R., Schmidt, K.:
{\em Positive definite kernels, continuous tensor products, and
  central limit theorems of probability theory}, volume 272 of {\em Lecture
  Notes in Math.},
Springer-Verlag, Berlin, 1972.

\bibitem{schuermann90}
Sch\"urmann, M.:
Gaussian states on bialgebras,
in {\em Quantum Probability V, Lecture Notes in Math.} {\bf 1442}, (1990) 347--367,  Springer-Verlag, Berlin.

\bibitem{schuermann93}
Sch{\"u}rmann, M.:
{\em White Noise on involutive Bialgebras}, volume 1544 of {\em Lecture Notes in Math.}, Springer-Verlag, Berlin, 1993.

\bibitem{skeide94}
Skeide, M.:
{\em The {L}{\'e}vy-{K}hintchine formula for the quantum group
  {$SU\sb q(2)$}},
PhD thesis, Univ. Heidelberg, Naturwiss.-Math. Gesamtfak., 1994.

\bibitem{skeide99}
Skeide, M.:
Hunt's formula for $SU_q(2)$ --- a unified view,
{\em Open Syst.\ Inf.\ Dyn.} {\bf 6} (1999) 1--27.

\bibitem{schuermann+skeide98}
Sch{\"u}rmann, M., Skeide, M.:
Infinitesimal generators of the quantum group {$SU_q(2)$}.
{\em Inf. Dim. Anal., Quantum Prob. and Rel. Topics} {\bf 1} (1998) 573--598.

\end{thebibliography}

\end{document}